\documentclass[11pt]{article}
\usepackage[table,xcdraw]{xcolor}
\usepackage{adjustbox}
\usepackage{geometry}
\usepackage[english]{babel}
\usepackage{amsmath}
\usepackage{amsthm}
\usepackage{natbib}
\usepackage[affil-it]{authblk}
\usepackage{graphicx}
\usepackage{comment}
\usepackage{amsfonts}
\usepackage{todonotes}
\newtheorem{example}{Example}

\newtheorem{theorem}{Theorem}
\newtheorem{definition}{Definition}
\newtheorem{lemma}{Lemma}
\usepackage{mathrsfs}
\usepackage{float}
\usepackage{booktabs}
\usepackage{multirow}
\usepackage{complexity}
\usepackage{subfig}
\usepackage{caption}
\usepackage{upgreek}
\usepackage{hyperref}
\hypersetup{
  colorlinks   = true,
  urlcolor     = cyan,
  linkcolor    = red,
  citecolor   = blue
}

\usepackage[ruled,vlined]{algorithm2e}
\usepackage{cases}
\usepackage{anyfontsize}

\makeatletter
\def\blfootnote{\xdef\@thefnmark{}\@footnotetext}
\makeatother

\makeatletter
\def\ps@pprintTitle{%
  \let\@oddhead\@empty
  \let\@evenhead\@empty
  \let\@oddfoot\@empty
  \let\@evenfoot\@oddfoot
}
\makeatother
\title{Branch-and-cut algorithms for colorful components problems\blfootnote{\textit{Email addresses:} \texttt{claudia.archetti@unibs.it} (Claudia Archetti), \texttt{mcerulli@unisa.it} (Martina Cerulli),
\texttt{csorgente@unisa.it} (Carmine Sorgente)}}
\date{}

\author[1]{Claudia Archetti}
\author[2]{Martina Cerulli}
\author[3]{Carmine Sorgente}
	
	\affil[1]{Department of Economics and Management, University of Brescia, Brescia (Italy)}
	
	\affil[2]{Department of Computer Science, University of Salerno, Fisciano (Italy)}

 	\affil[3]{Department of Mathematics, University of Salerno, Fisciano (Italy)}

\begin{document}
\maketitle
\vspace*{-1cm}
\hspace*{-6mm}\fcolorbox{red}{white}{\parbox{\textwidth}{This paper has been accepted for publication in the \textbf{INFORMS Journal on Computing}. The \textbf{final published version} is available at \url{https://doi.org/10.1287/ijoc.2024.0927}, along with its supplemental material including an online appendix and a software data repository (\url{https://github.com/INFORMSJoC/2024.0927}).}}
\vspace*{4mm}

\begin{abstract}
We tackle three optimization problems in which a colored graph, where each node is assigned a color, must be partitioned into colorful connected components. A component is defined as colorful if each color appears at most once. The problems differ in the objective function, which determines which partition is the best one. These problems have applications in community detection, cybersecurity, and bioinformatics. We present integer non-linear formulations, which are then linearized using standard techniques. To solve these formulations, we develop exact branch-and-cut algorithms, embedding various improving techniques, such as valid inequalities, bounds limiting the number of variables, and warm-start and preprocessing techniques. Extensive computational tests on benchmark instances demonstrate the effectiveness of the proposed procedures. The branch-and-cut algorithms can solve reasonably sized instances efficiently. To the best of our knowledge, we are the first to propose an exact algorithm for solving these problems.
\end{abstract}

\section{Introduction}
Graph theory serves as a fundamental framework for modeling complex systems in various domains, including computer science, social networks, cybersecurity, biology, and transportation systems. Within this rich mathematical field, the study of connected components has played an important role in understanding the structural properties and dynamics of graphs. The connected components or simply the \textit{components} of a graph are subgraphs where each node can be reached from every other node in the subgraph via a path. 

An interesting subclass of problems related to connected components is the one in which nodes are colored. 
Specifically, given a node-colored graph~$\mathcal{G}$, with node set $V$ and edge set $E$, any connected component of $\mathcal{G}$ is said to be \textit{colorful} if all its nodes have different colors.
This paper addresses three problems related to these colorful components of a graph: the ``Minimum Orthogonal Partition'' (MOP) problem, also referred to as ``Colorful Components'' in \cite{bruckner2012}, the ``Maximum Edges in transitive Closure'' (MEC) problem, and the ``Minimum Colorful Components'' (MCC) problem. 

The MOP, MEC and MCC problems have been introduced in the context of orthology gene identification in bioinformatics \citep{zheng2011,bruckner2012,popa2014}, where different colors are associated with genes from different genomes linked by pairwise homology relationships, and the so-called \textit{homology graph} has to be converted into a new graph where spurious homologies are removed, with each component satisfying the orthogonality property.
Specifically, given a node-colored graph, all three problems aim at removing edges in such a way that all the connected components of the resulting graph are colorful. However, they differ in terms of objective function: the MOP problem aims at minimizing the number of edges removed; in the MEC problem, the objective is to maximize the transitive closure of the resulting graph; in the MCC the aim is to minimize the number of resulting colorful components.

In addition to applications in bioinformatics, the MOP, MEC and MCC problems arise in various other fields. In social networks, where nodes represent individuals and edges represent connections (friendships, interactions, common interests), all three problems can be used to determine the most influential connections linking distinct (colorful) and cohesive (connected) communities. By guaranteeing the colorfulness of the communities through the removal of specific friendship edges, as in the MOP problem, a social network aims to mitigate the risk of echo chambers, where users predominantly interact with similar individuals. Furthermore, maximizing the number of edges in the transitive closure in the MEC problem means maximizing transitively closed relationships within the network. This promotes the propagation of information through the community, ensuring that a node can be reached by others through a sequence of edges, avoiding the isolation of users within small disconnected groups.
Instead, by minimizing the number of colorful components, the MCC problem helps prevent the dispersion of users into too many groups.
When considering cyberspace networks of computers and devices with various types of connections (e.g., permissions, trust levels, data flow), identifying a subset of edges to be removed while ensuring that the remaining network is composed of colorful components helps optimize the network’s resilience to cyber threats. In the context of blockchains, the Colorful Components problems can be seen as an analogy for sharding \citep{liu2023}. It consists in splitting the network into subnetworks (shards), which manages a portion of the transactions or states, with the goal of improving scalability and security. If the graph represents a blockchain where each node has a specific role (e.g., validators, storage, execution of smart contracts) corresponding to a color, partitioning the nodes into colorful components ensures that the nodes within a shard have distinct roles. This optimizes the functioning of the shards, as each component performs a specific set of functions without overlap, improving load distribution.

\paragraph{Related works} 
The colorful components problems have been studied in comparative genomics \citep{zheng2011}, a branch of bioinformatics dedicated to exploring the structural relationships of genomes across distinct biological species. In this framework, colorful graphs serve as representations of similarities among genes belonging to various homologous gene families: if two nodes (genes) are connected by an edge, those genes have a certain level of similarity or homology; if two nodes share the same color, they belong to the same genome. The concept of ``colorful components'' involves dividing the graph into distinct sections or partitions. Each partition, referred to as a colorful component, corresponds to an orthology set, i.e., a collection of genes that are evolutionarily related, typically stemming from a common ancestor. The partitioning ensures that genes from the same genome are placed into different orthology sets, emphasizing diversity and evolutionary distinctions.

The MOP problem has been introduced in \cite{he2000}. As noted in \cite{bruckner2012}, it can be seen as the problem of destroying, by edge removals, all the paths between two nodes of the same color. In this sense, it is a special case of the \NP-hard Minimum Multi-Cut problem, which, given a set of node pairs (in the MOP case, pairs of nodes having the same color), asks for the minimum number of edges to be removed from the graph to disconnect each given node pair. 
It is also a special case of the Multi-Multiway Cut Problem \citep{avidor2007} which, given some node sets (in the MOP case, sets of nodes having the same color), aims to find the minimum edge set whose removal completely disconnects all node sets.
In \cite{bruckner2012}, it is shown that the MOP problem is polynomial-time solvable for two or fewer colors and \NP-hard otherwise. Fixed-parameter algorithms are also discussed: it is shown that the MOP problem is fixed-parameter-tractable for general colored graphs when parameterized by the number of colors and the number of edge deletions. In \cite{misra2018}, the size of a node cover is considered as the parameter. In \cite{he2000}, an approximation algorithm is proposed for solving the MOP problem on an edge-weighed graph.
Heuristic approaches are proposed in \cite{zheng2011} and \citep{bruckner2012}. In \cite{bruckner2013}, an application for correcting Wikipedia interlanguage links is proposed. These links often have errors due to manual updates or na\"ive bots, and these errors may be found through a graph model \citep{de2010untangling}: each word in a language corresponds to a node, and an interlanguage link corresponds to an edge. The goal is to partition the graph such that each connected component corresponds to a term in multiple languages, ensuring each language appears at most once in each component. In this paper, besides proposing two heuristics for the MOP problem, the authors also solve it as an implicit hitting set problem and a clique partition problem. A hitting set is a set of edges that intersects with every bad path set (a cycle-free path between two nodes of the same color) in a collection of bad path sets. The hitting set problem then aims to find the smallest subset of edges (hitting set) so that removing these edges resolves all violations. \cite{bruckner2013} use the implicit hitting set framework \citep{chandrasekaran2011algorithms,moreno2013}, allowing for dynamically generating constraints (sets) in the MOP problem, i.e., instead of generating all sets upfront, the algorithm starts with a small subset and iteratively adds more sets (constraints) as needed. Instead, the clique partition-based ILP formulation \citep{grotschel1989cutting} transforms the problem into finding a partition of the graph into cliques, ensuring the colorful property is maintained. It has only polynomially many constraints, as opposed to the implicit hitting set formulation which has exponentially many constraints. However, the number of constraints may be too large, therefore, the authors implement a row generation scheme.

Assuming that the orthologous genes trace back to a common ancestor, it is clear that the orthology relation between these genes exhibits transitivity: if gene A is orthologous to gene B, and gene B is orthologous to gene C, then gene A is also orthologous to gene C. This motivates the study of the MEC problem, where transitivity is modeled with transitive closure. 
In \cite{zheng2011}, the MEC problem is conjectured to be \NP-hard. In \cite{popa2014}, it is proved to be \APX-hard when the number of colors in the graph is at least $4$. The authors show the result via a reduction from the MAX-3SAT problem. 
In \cite{popa2015}, the MEC problem is proven to be \APX-hard even in the case when the number of colors is $3$ and \NP-hard to approximate within a factor of $|V|^{(1/3-\epsilon)}$, for any $\epsilon > 0$, when the number of colors is arbitrary, even when the input graph is a tree where each color appears at most twice. A heuristic to solve the MEC problem is presented in \cite{zheng2011},
while \cite{popa2015} present a polynomial-time approximation algorithm. 
In \cite{Dondi20181}, the parameterized and approximation complexity of MCC and MEC problems, for general and restricted instances, is investigated.

The MCC problem is introduced in \cite{popa2014} where the authors prove that it does not admit polynomial-time approximation within a factor of $|V|^{\frac{1}{14}-\epsilon}$, for any $\epsilon > 0$, unless \P =\NP, even if each node color appears at most twice. It is shown by \cite{Dondi20181} that the problem is equivalent to the Minimum Multi-Cut problem on trees \citep{hu1963}. Indeed, when considering a tree, the MCC problem coincides with the MOP problem (since the number of removed edges is equivalent to the number of obtained colorful components), which, as already discussed, can be traced back to the Minimum Multi-Cut problem. Because of this equivalence on trees, the MCC problem is not approximable within factor $1.36 - \epsilon$ for any $\epsilon>0,$ is fixed-parameter tractable, and admits a poly-kernel (when the parameter is the number of colorful components). Moreover, it is shown that the MCC problem is polynomial-time solvable on paths, while it is \NP-hard even for graphs with a distance of 1 to the class of disjoint paths.

The MOP, MEC, and MCC problems belong to the class of \textit{graph modification problems} \citep{Sritharan201632}, which consist in performing a set of modifications to the node and/or edge sets of a graph in order to satisfy specified properties.
Well-known problems in this class aim to produce chordal graphs \citep{Yannakakis1981, Natanzon20011}, planar graphs \citep{Yannakakis1978}, interval graphs \citep{Benzer1959}, cluster graphs \citep{Shamir2004, Ambrosio2025}, as well as to reduce as much as
possible the size of a given combinatorial structure of the graph \citep{zenklusen2010matching, furini2020, Wei2021, cerulli2023mathematical}. 
In colorful components problems, the allowed modifications are edge deletions, while the property that the final graph must satisfy is being a set of colorful components.

\paragraph{Contributions}
To the best of our knowledge, we are the first to formulate the MOP, MEC, and MCC problems as integer nonlinear problems. We linearize the formulations and propose valid inequalities, warm-start, and preprocessing procedures to enhance them. We further provide a formulation to determine the maximum-cardinality colorful component, which is used to derive bounds on the cardinality of the optimal colorful components set. Branch-and-cut algorithms are implemented to solve the formulations, with dynamic separation of the exponentially many connectivity constraints. Computational tests are performed on benchmark and randomly generated instances. 
The results show that the configurations that use the valid inequalities, the preprocessing procedure and warm-start algorithms significantly outperform the plain model by reducing runtime and increasing the number of instances solved to optimality.
\paragraph{Structure of the paper} The paper is organized as follows. In Section~\ref{sec:def} we provide problem definitions and formulations. In Section~\ref{sec:properties} the formulation of the maximum-cardinality colorful component problem is proposed, together with tighter upper bounds on the number of colorful components in any optimal partition of the graph. Different algorithms are presented to compute these upper bounds and find a warm-start solution for the branch-and-cut algorithm. The overall branch-and-cut algorithm is presented in Section~\ref{sec:bec}, where valid inequalities that strengthen the formulations are introduced as well. Section~\ref{sec:results} is devoted to the numerical experiments, and Section~\ref{sec:conclusion} concludes the paper.

\section{Definitions and formulations}\label{sec:def}
In this Section, we first give the formal definition of the problems in Section \ref{sec:def1} and then provide the corresponding mathematical formulations in Section~\ref{sec:form}.

\subsection{Problems definitions}\label{sec:def1}
Before providing a formal definition of the three problems, let us define the \textit{transitive closure} of a graph, which describes the connectedness of its nodes. Specifically, the transitive closure of an undirected graph $\mathcal{G}$ is a graph $\mathcal{H}=(V,E_H)$, where $E_H=\{\{i,j\}: i,j \in V,\ \text{$i$ is connected to $j$ in $\mathcal{G}$}\}$. In other words, it is a \textit{cluster graph} where the nodes of each component form a clique.
Moving to problem definitions, we are given a node-colored graph $\mathcal{G}=(V, E, C)$, where 
$C$ is the set of colors associated with the nodes in $V$, $c_u$ is the color of node $u \in V$, and $V^c\subset V$ the set of nodes $u \in V$ having color $c_u = c$. For a given set of nodes $K \subseteq V$, let $\mathcal{G}[K]$ denote the subgraph of $\mathcal{G}$ induced by $K$. Given a set of nodes $S$, we denote as \textit{colorful} any connected component $\mathcal{G}[S]$ of $\mathcal{G}$ such that all the nodes in $S$ have a different color, i.e., $c_u \neq c_v$ for any pair of nodes $u,v \in S$.
Any partition of $\mathcal{G}$ into colorful components is a feasible solution for all three problems. However, they differ in terms of the objective. The formal definition of each problem is as follows.
\begin{definition}[MOP Problem]
    Given a node-colored graph $\mathcal{G} = (V,E,C)$, the MOP problem consists in finding the smallest subset of edges $E' \subseteq E$ to remove from the graph $\mathcal{G},$ such that in the resulting graph $\mathcal{G}'$, with node set $V$ and edge set $E\setminus E'$, all the connected components are colorful.
\end{definition}
\begin{definition}[MEC Problem]
     Given a node-colored graph $\mathcal{G} = (V,E,C)$, the MEC problem consists in finding the subset of edges $E' \subseteq E$ to remove from the graph $\mathcal{G},$ such that in the resulting graph $\mathcal{G}'$, with node set $V$ and edge set $E\setminus E'$, all the connected components are colorful and the number of edges in its transitive closure is maximized.
\end{definition}
\begin{definition}[MCC Problem]
     Given a node-colored graph $\mathcal{G} = (V,E,C)$, the MCC problem consists in finding the subset of edges $E' \subseteq E$ to remove from the graph $\mathcal{G},$ such that the resulting graph $\mathcal{G}'$, with node set $V$ and edge set $E\setminus E'$, consists of the smallest number of colorful components. 
\end{definition}

Despite sharing the same set of feasible points, the three problems might differ in terms of the optimal solution, as shown in the following examples. 
\begin{figure}[H]
  \caption{\textbf{Comparison between the MOP and MEC problems' objectives.}}
  \label{fig:MOPvsMEC}
   \centering
   \hspace*{\fill}%
   \hfill
   \subfloat[Node-colored graph $\mathcal{G}_1$.]{\includegraphics[scale=0.59]{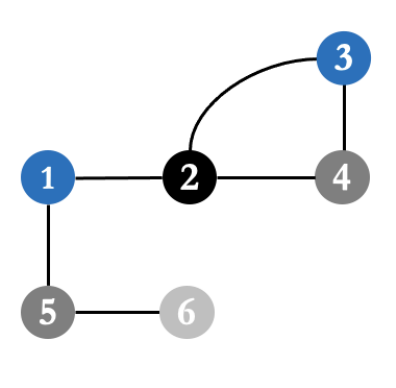}\label{fig:MOPvsMEC-graph}}\hfill   \hfill
   \subfloat[MOP solution for $\mathcal{G}_1$.]{\includegraphics[scale=0.59]{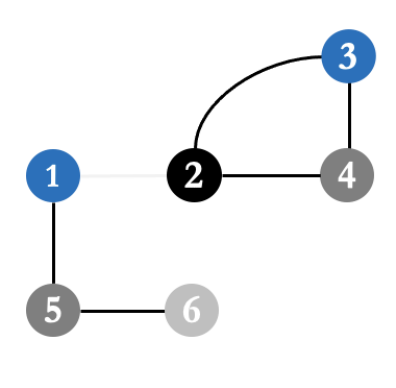}\label{fig:MOPvsMEC-MOPopt}}\hfill   \hfill
  \subfloat[MEC solution for $\mathcal{G}_1$.]{\includegraphics[scale=0.59]{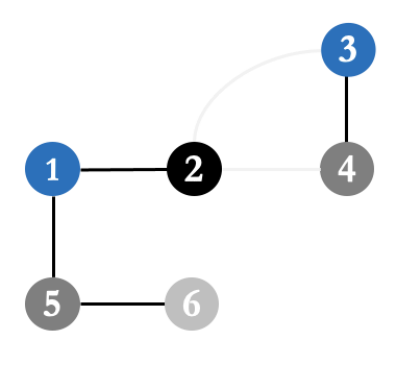}\label{fig:MOPvsMEC-MECopt}}\hfill
  \hspace*{\fill}%
\end{figure}
\begin{example}{}{}
    Consider the graph $\mathcal{G}_1$ depicted in Figure~\ref{fig:MOPvsMEC-graph}, where nodes $1$ and $3$, as well as $4$ and $5$, have the same color.
    An optimal solution to the MOP problem consists in removing edge $\{1,2\}$ only, producing two colorful components, as shown in Figure~\ref{fig:MOPvsMEC-MOPopt}. 
    The transitive closure associated with such a partition contains $6$ edges, while the optimal value of the MEC problem is $7$, as testified by the solution shown in Figure~\ref{fig:MOPvsMEC-MECopt}, obtained by removing two edges, i.e., $\{2,4\}$ and $\{2,3\}$.
\end{example}
\begin{figure}[H]
  \caption{\textbf{Comparison between the MOP and MCC problems' objectives.}}
  \label{fig:MCCvsMOP}
   \centering
   \hspace*{\fill}%
   \hfill
   \subfloat[Node-colored graph $\mathcal{G}_2$.]{\includegraphics[scale=0.59]{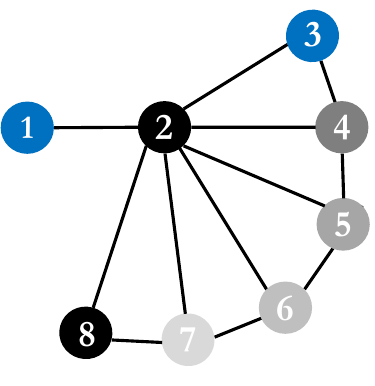}\label{fig:MCCvsMOP-graph}}\hfill\hfill
  \subfloat[MOP solution for $\mathcal{G}_2$.]{\includegraphics[scale=0.59]{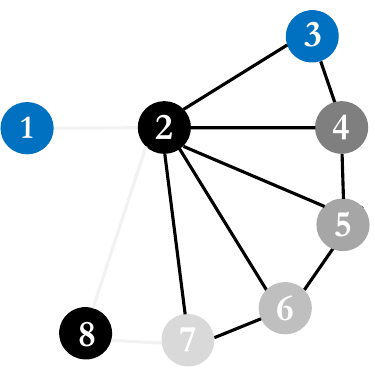}\label{fig:MCCvsMOP-MOPopt}}\hfill\hfill
  \subfloat[MCC solution for $\mathcal{G}_2$.]{\includegraphics[scale=0.59]{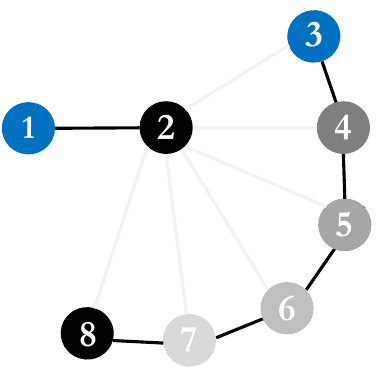}\label{fig:MCCvsMOP-MCCopt}}\hfill
  \hspace*{\fill}%
\end{figure}
\begin{example}
    Consider the graph $\mathcal{G}_2$ depicted in Figure~\ref{fig:MCCvsMOP-graph}, where nodes $1$ and $3$, as well as nodes $2$ and $8$, have the same color. 
    On the one hand, the smallest set of edges whose removal partitions  $\mathcal{G}_2$ into colorful components has size three. The corresponding optimal solution for the MOP problem with such a value is shown in Figure~\ref{fig:MCCvsMOP-MOPopt} and consists of three components. On the other hand, this partition does not lead to the smallest number of components. Indeed, removing the six light-grey colored edges in Figure~\ref{fig:MCCvsMOP-MCCopt} isolates nodes $1$ and $2$ from the remaining ones, producing only two colorful components.
    \end{example}
\begin{figure}[H]
  \caption{\textbf{Comparison between the MEC and MCC problems' objectives.}}
  \label{fig:MCCvsMEC}
   \centering
   \hspace*{\fill}%
   \hfill
   \subfloat[Node-colored graph $\mathcal{G}_3$.]{\includegraphics[scale=0.59]{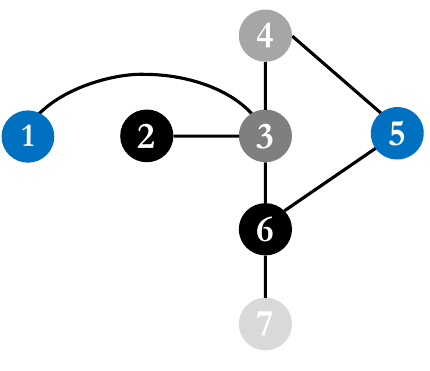}\label{fig:MCCvsMEC-graph}}\hfill\hfill
   \subfloat[MEC solution for $\mathcal{G}_3$.]{\includegraphics[scale=0.59]{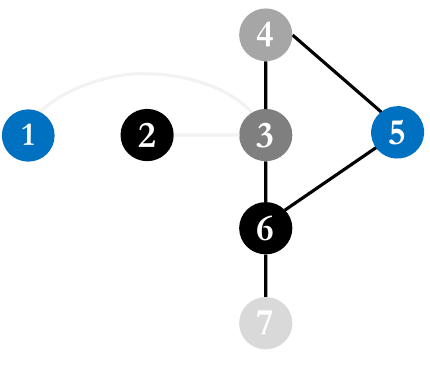}\label{fig:MCCvsMEC-MECopt}}\hfill\hfill
   \subfloat[MCC solution for $\mathcal{G}_3$.]{\includegraphics[scale=0.59]{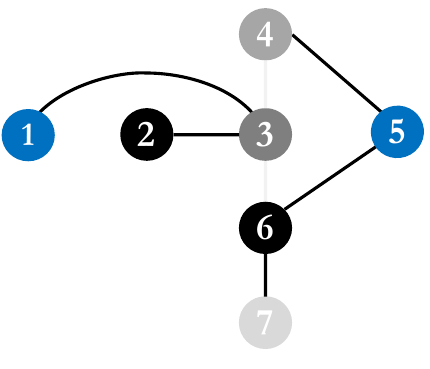}\label{fig:MCCvsMEC-MCCopt}}\hfill
  \hspace*{\fill}%
\end{figure}
\begin{example}
    Consider the graph $\mathcal{G}_3$ depicted in Figure~\ref{fig:MCCvsMEC-graph}, where nodes $1$ and $5$, as well as nodes $2$ and $6$, have the same color. Maximizing the number of edges in the transitive closure of the resulting graph leads to a partition of $\mathcal{G}_3$ into three colorful components, as shown in Figure~\ref{fig:MCCvsMEC-MECopt}. Such a solution is associated with a transitive closure containing $10$ edges. If one aims at minimizing the number of components, instead, the two light-grey colored edges in Figure~\ref{fig:MCCvsMEC-MCCopt} would be removed, obtaining two colorful components of size three, whose transitive closure contains $9$ edges.
    \end{example}

\subsection{Mathematical formulations}\label{sec:form}
Let us denote by $\mathcal{Q}$ the set of colorful components into which the nodes of $\mathcal{G}$ will be partitioned, with cardinality $Q=|\mathcal{Q}|$. A trivial upper bound on $Q$, satisfied by every feasible partition of $\mathcal{G}$ into colorful components, is the number of nodes $|V|$. As we will show in Section~\ref{sec:properties}, tighter bounds on the number of components in any optimal solution can be derived specifically for each problem.

In order to formulate the MOP, MEC, and MCC problems, we define the following binary variables:
\begin{itemize}
    \item  $x_{i}^{k}$, for each $i \in V$ and for each $k \in \mathcal{Q}$, s.t.\ $x_{i}^{k}=1$ iff node $i$ is in the $k$-th component;
    \item $y_{ij}$, for each $\{i,j\} \in E$, s.t.\ $y_{ij}=1$ iff $\{i,j\}$ is not removed from the graph, i.e., $\{i,j\} \notin E'$.
\end{itemize}

A partition of $\mathcal{G}$ into colorful components, resulting from the removal of a subset of edges $E' \subseteq E$ from $\mathcal{G}$, has to satisfy two conditions: (i) for any pair of nodes $i,j \in V$, if $i$ and $j$ are assigned to the same component, they must be \textbf{connected} by means of the edges in $E \setminus E'$, and (ii) in any component, each \textbf{color} in $C$ appears at most once.
These conditions lead to the definition of a feasible set $\mathcal{F}$ for variables ${x}$ and ${y}$ satisfying the following constraints:
{\begin{subequations}
    \label{feasible-set-F}
    \begin{align}
        (\mathcal{F}) \hspace{6ex} & \sum\limits_{k \in \mathcal{Q}} x_{i}^{k} = 1 & \quad \forall\; i \in V\label{assignment}\\
        &\sum\limits_{\substack{i \in V^{c}}} x_{i}^{k} \le 1 &\quad \forall\; c \in C,\, k \in \mathcal{Q}\label{color}\\
        &{|U||V\setminus U|}\cdot\sum\limits_{\{u,v\} \in \delta(U)} y_{uv} \ge \sum\limits_{i \in U}\sum\limits_{j \notin U} \sum\limits_{k \in \mathcal{Q}} x_{i}^{k}x_{j}^{k} &\quad \forall\; U \subset V
        \label{connectivity}\\
        &y_{ij} \le \sum\limits_{k\in \mathcal{Q}} x_{i}^{k}x_{j}^{k} &\quad \forall\; \{i,j\} \in E\label{y_vars}\\
        &x \in \{0,1\}^{|V|\times Q}, y \in \{0,1\}^{|E|}, &
    \end{align}
\end{subequations}}
with $\delta(U)$ denoting the edges with exactly one endpoint in $U$. 

Constraints~\eqref{assignment} force the assignment of each node to exactly one component.
Constraints~\eqref{color} guarantee the occurrence of each color at most once in each component.
Constraints~\eqref{connectivity} link the $y$ and $x$ variables, enforcing connectivity inside each component: for each subset $U \subset V$, if at least one node $i\in U$ belongs to the same component $k$ as a node $j\notin U$
(i.e., $x_{i}^{k}x_{j}^{k} = 1$ with $i \in U$ and $j \in V\setminus U$), the number of non-removed edges in $\delta(U)$ has to be at least one, otherwise $i$ and $j$ would be disconnected. 
Finally, constraints~\eqref{y_vars} force the removal of the edge $\{i,j\}$ if $i$ and $j$ have not been assigned to the same component, as well as, the other way round, the existence of a component containing both $i$ and $j$ if the edge $\{i,j\}$ is not removed from the graph, i.e., if $y_{ij} = 1$ there exists $k$ such that $x_{i}^{k}x_{j}^{k} = 1$.

The nonlinearity in the feasible set~\eqref{feasible-set-F} is given by the bilinear product $x_{i}^{k}x_{j}^{k}$, which appears in the constraints~\eqref{connectivity}--\eqref{y_vars}. To deal with this nonlinearity, we use the Fortet reformulation approach (it corresponds to a specific version of McCormick reformulation \citep{mccormick1976} that deals with products between binary variables). It consists in defining an auxiliary variable $z_{ij}^k$ for each $i,j\in V$ and $k \in \mathcal{Q}$, representing the product $x_{i}^{k}x_{j}^{k}$, and adding to the feasible set~\eqref{feasible-set-F} constraints~\eqref{mccormick}, obtaining a new feasible set $\mathcal{F}^\ell$ defined by the following constraints:
\begin{subequations}
    \label{feasible-set-F_lin}
    \begin{align}
        (\mathcal{F}^\ell) \hspace{6ex} &\; \sum\limits_{k \in \mathcal{Q}} x_{i}^{k} = 1 & \forall\; i \in V\label{assignment_lin}\\
        &\; \sum_{\substack{i \in V^{c}}} x_{i}^{k} \le 1 & \forall\; c \in C, \, k \in \mathcal{Q}\label{color_lin}\\
        &\; {|U||V\setminus U|}\cdot\sum_{\{u,v\} \in \delta(U)} y_{uv} \ge \sum_{i \in U}\sum_{j \notin U} \sum\limits_{k \in \mathcal{Q}} z_{ij}^k & \forall\; U \subset V\label{connectivity_lin}\\
        &\; y_{ij} \le \sum\limits_{k\in \mathcal{Q}} z_{ij}^k & \forall\; \{i,j\} \in E\label{y_vars_lin} \\
        &\; z_{ij}^k \ge x_{i}^{k} + x_{j}^{k} -1,  \quad z_{ij}^k \le x_{i}^{k}, \quad z_{ij}^k \le x_{j}^{k}  & \forall\; i,j \in V
        ,\ k \in \mathcal{Q}\label{mccormick}\\
        &\; x \in \{0,1\}^{|V|\times Q}, y \in \{0,1\}^{|E|}, z \in [0,1]^{|V|\times|V|\times Q}. &
    \end{align}
\end{subequations}
Constraints are the same defining $\mathcal{F}$ with the addition of \eqref{mccormick}.

Note that \textit{connectivity constraints}~\eqref{connectivity} can be disaggregated and formulated as
\begin{equation}\label{connectivity:disag}
    \quad \sum_{\{u,v\} \in \delta(U)} y_{uv} \ge x_{i}^{k}x_{j}^{k} \qquad \forall\; U \subset V, i \in U, j \notin U, k \in \mathcal{Q},
\end{equation}
and the linearized version~\eqref{connectivity_lin} as 
\begin{equation}\label{connectivity:disag_lin}
 \quad \sum_{\{u,v\} \in \delta(U)} y_{uv} \ge z_{ij}^{k} \qquad \forall\; U \subset V, i \in U, j \notin U, k \in \mathcal{Q},
\end{equation}
obtaining a tighter formulation but with many more constraints.
On the contrary, the $|E|$ \textit{edge constraints}~\eqref{y_vars} can be replaced by the following exponentially many \textit{path constraints}:
\begin{align}
& \sum_{\{u,v\}\in P_{ij}} y_{uv} \leq \left(|P_{ij}|-1\right)+\sum_{k \in \mathcal{Q}}x_i^kx_j^k & \quad \forall\ P_{ij}\,:\,i,j\in V,\, |P_{ij}| \ge 1, \label{path-ineqs}
\end{align}

\noindent with $P_{ij}$ being a path connecting nodes $i$ and $j$ through at least one edge (the case of a path with one edge only corresponds to inequalities~\eqref{y_vars}). For each of these paths, if nodes $i$ and $j$ belong to different components, the number of edges along the path that remain in the graph is constrained to be at most $|P_{ij}|-1$. We remark that constraints~\eqref{y_vars} imply these inequalities. Indeed, along with any path connecting $i$ and $j$, there exists a node $u$ which belongs to a different component w.r.t.\ its neighbor $v$ in the path, and, because of constraint~\eqref{y_vars} related to this pair of nodes, the edge $\{u,v\}$ is removed. 
In the linearized version, constraints \eqref{path-ineqs} are replaced by 
\begin{equation}
    \quad \sum_{\{u,v\}\in P_{ij}} y_{uv} \leq \left(|P_{ij}|-1\right)+\sum_{k \in \mathcal{Q}}z_{ij}^k \qquad \forall\ P_{ij}\,:\,i,j\in V,\, |P_{ij}| \ge 1. \label{path-ineqs_lin}
\end{equation}

\subsubsection{MOP Formulation}\label{subsec:MOP}
The MOP problem asks for the minimum number of edges to be removed from the graph $\mathcal{G}$ to obtain a partition into colorful components. 
A nonlinear binary formulation of the MOP problem is the following:
\begin{equation}\label{MOP}
    \max\limits_{(x,\,y)\in \mathcal{F}}\; \sum_{\{i,j\} \in E} y_{ij}.
\end{equation}
The objective function of~\eqref{MOP} maximizes the number of edges remaining in the graph, by summing up the values of the variables $y$. The feasible set $\mathcal{F}$ of variables $x$ and $y$ is described in~\eqref{feasible-set-F}. Its linearized version, involving the additional variables $z$, can be obtained by considering the feasible set $\mathcal{F}^\ell$ in~\eqref{feasible-set-F_lin} instead of $\mathcal{F}$. We notice that, contrary to the MEC and the MCC problems, in the MOP problem we do not need to guarantee the connectivity of all nodes assigned to connected component $k \in \mathcal{Q}$. Thus, connectivity constraints~\eqref{connectivity} (or \eqref{connectivity_lin}) are not necessary. However, they correspond to optimality cuts. In Section~\ref{sec:results}, we show their impact on the solution of the MOP.

\subsubsection{MEC Formulation}\label{subsec:MEC}%
The MEC and MOP problems share the same feasible set $\mathcal{F}$, imposing that a partition into colorful components of the graph must be obtained. However, while the MOP problem aims to minimize the number of removed edges, in the MEC problem the number of edges in the transitive closure is maximized. 
A nonlinear formulation for the MEC problem thus reads:
\begin{equation}\label{MEC}
    \max\limits_{(x,\,y) \in \mathcal{F}} \;\sum\limits_{k \in \mathcal{Q}} \sum_{i \in V} \sum_{j \in V : i<j} x_{i}^{k}x_{j}^{k}.
\end{equation}
The objective function of~\eqref{MEC} sums up, for each component $k$, and for each pair of nodes in $V$, the products $x_{i}^{k}x_{j}^{k}$, which represents the number of edges in the transitive closure of component $k$. Indeed, if $i$ and $j$ belong to the same component $k$ (the product is 1), they are necessarily connected, i.e., there exists an edge between $i$ and $j$ in the transitive closure. Again, by introducing the auxiliary variables $z$, the feasible set $\mathcal{F}^\ell$ involving only linear constraints can be considered instead of the set $\mathcal{F}$. The objective function will consequently read as $\sum\limits_{k \in \mathcal{Q}} \sum\limits_{i \in V} \sum\limits_{j \in V : i<j} z_{ij}^{k}.$

\subsubsection{MCC formulation}\label{subsec:MCC}
To formulate the MCC problem, in addition to the already defined variables $x$ and $y$ (and $z$), we introduce a family of binary variables $w_{k}$, for each $k \in \mathcal{Q}$, equal to 1 iff at least one node has been assigned to the $k$-th component. 
A nonlinear formulation for the MCC problem reads:\\ 
\begin{minipage}{\textwidth}
\begin{subequations}\label{MCC}
\begin{align}
    \min_{x,y,w} &\; \sum\limits_{k \in \mathcal{Q}} w_{k} &\label{MCC:obj}\\
    \text{s.t.} & \; |V| \cdot w_k \ge \sum_{i \in V} x_{i}^{k} & \forall\; k \in \mathcal{Q}\label{MCC:w_vars}\\
    &\; (x,\,y) \in \mathcal{F}, w \in \{0,1\}^{Q}. &
\end{align}
\end{subequations}
\vspace{0.05mm}
\end{minipage}
The objective function~\eqref{MCC:obj} minimizes the number of components, by summing up the values of auxiliary variables $w$. The value of variables $w$ is properly set by constraints~\eqref{MCC:w_vars}: for each $k \in \mathcal{Q}$, if there exists at least one $i \in V$ s.t.\ the related $x_{i}^{k}$ variable assumes value one, the right-hand side of the constraint is strictly greater than zero, forcing $w_k$ to be equal to one. 
The remaining constraints are defined by the set $\mathcal{F}$, described in Eq.~\eqref{feasible-set-F}. Equivalently, we can employ Fortet's reformulation to eliminate the bilinear terms, i.e., introduce variables $z$ and replace $\mathcal{F}$ with $\mathcal{F}^\ell$.

\section{Related problems and bounds on the number of colorful components}\label{sec:properties}
All formulations proposed in Sections~\ref{subsec:MOP}, \ref{subsec:MEC} and \ref{subsec:MCC} depend on the size of $\mathcal{Q}$, i.e., $Q$. A trivial upper bound is $Q=|V|$. However, one can reduce the size of the formulations by tightening the value of $Q$. To this aim, for the MOP problem, in Section~\ref{subsec:pairs}, we present a formulation to find the maximum number of disjoint colorful pairs (components consisting of exactly two nodes of different colors) which is then used to determine a bound on the number of colorful components in any MOP optimal solution. For the MEC and MCC problems, we instead use the notion of maximum-cardinality colorful component introduced in Section~\ref{subsec:maxcard}.
In the following, we denote as $\bar{Q}$ the upper bound on the number of colorful components in any optimal partition of the graph, whatever the problem considered.

\subsection{Maximum number of disjoint colorful pairs}\label{subsec:pairs}
We provide here a mathematical formulation for the problem of determining the maximum number of disjoint colorful pairs. This number is used in the following to derive 
a valid bound $\bar{Q}$ for the MOP problem, as described in Section~\ref{subsec:Q3}.
An integer model to compute the maximum number of disjoint colorful pairs of $\mathcal{G}$ can be formulated using a binary variable $Y_{ij}$ which is $1$ if and only if edge $\{i,j\}$ is selected, for each $\{i,j\} \in E$. Let us define $A^c_i$ for all $i \in V$ and $c\in C$ as a binary parameter which is $1$ if $c_i = c$ and $0$ otherwise. The formulation is given below.\\
\begin{minipage}{\textwidth}
\begin{subequations}\label{pairs}
    \begin{align}
    \max\limits_{Y} &\sum\limits_{\{i,j\} \in E} Y_{ij} &\\
    \text{s.t.} & \;\sum_{\{i,j\} \in \delta(i)} Y_{ij} \leq 1 & \forall\ i \in V \label{MIP-pairs-delta_i}\\
    & \;Y_{ij} \leq 1 - \sum_{c \in C} A_i^c  A_j^c & \forall\ \{i,j\} \in E \label{MIP-pairs-Y}\\
    & \; Y\in \{0,1\}^{|E|}, & 
  \end{align}
\end{subequations}\vspace{0.05mm}
\end{minipage}
with $\delta(i)$ the set of edges $\{i,j\}$ for all $j \in V$, i.e., the edges having $i$ as one endpoint.
Constraints~\eqref{MIP-pairs-delta_i} allow for the selection of at most one edge incident on each node, while constraints~\eqref{MIP-pairs-Y} prevent the selection of all the edges linking nodes of the same color.
This model can be used to derive a valid upper bound on the number of colorful components of any MOP optimal solution, as described in Section~\ref{subsec:Q3}.

\subsection{Tighter bound 
on the number of colorful components in a MOP optimal solution}\label{subsec:Q3}
A valid upper bound on the number of colorful components for any MOP optimal solution can be obtained by determining the maximum number of disjoint colorful pairs, as stated in the following theorem. 
\begin{theorem}
\label{th:bound-Q-MOP}
    Being $\left\{S_i\right\}_{i=1}^k$ a collection of $k$ disjoint colorful pairs of $\mathcal{G}$, $\bar{Q}= |V|-k$ is an upper bound on the number of colorful components in any MOP optimal solution.
\end{theorem}
\begin{proof}{Proof.}
Proof in Online Appendix~A.
\end{proof}

The procedure to compute such a bound is shown in Algorithm~\ref{algo:heu3}. A feasible solution for the MOP problem is constructed by solving formulation~\eqref{pairs}, getting the corresponding set $\{S_i\}_{i=1}^k$ of colorful components, and then including the remaining singletons. This solution can be used as a warm-start to solve exactly the MOP problem. The corresponding number of components gives an upper bound on the number of colorful components in any MOP optimal solution, i.e., a valid $\bar{Q}$.
\setlength{\intextsep}{2pt}
\begin{algorithm}[ht]\caption{Computing a feasible solution and $\bar{Q}$ for the MOP problem}
    \LinesNumbered
    \SetAlgoLined 
    \label{algo:heu3}
    \KwData{Graph $\mathcal{G}$.} 
    Solve formulation~\eqref{pairs}, obtaining a set of $k$ colorful pairs $\mathcal{\bar{S}} = \left\{S_i\right\}_{i=1}^k$ of $\mathcal{G}$.\\
    Set $\bar{Q}=|V|-k.$\\
    \lForEach{$u \in V \setminus \left\{\bigcup_{i=1}^{k} S_i\right\}$}{
        Set $S = \{u\}$ and $\bar{\mathcal{S}} = \bar{\mathcal{S}} \cup \{S\}.$\DontPrintSemicolon}
    \Return $\bar{\mathcal{S}},\bar{Q}$
\end{algorithm}
\subsection{Maximum-cardinality colorful component}\label{subsec:maxcard}
The colorful component of $\mathcal{G}$ of maximum cardinality can be identified by solving the following formulation, involving two sets of binary variables: for each node $i\in V$, the binary variable $X_i$, which is $1$ if node $i$ is part of the maximum-cardinality colorful component, $0$ otherwise; for each edge $\{i,j\} \in E$, the binary variable $Y_{ij}$ which is $1$ if edge $\{i,j\}$ is selected, $0$ otherwise. 
The problem can be formulated as follows:

\begin{minipage}{\textwidth}
\begin{subequations}\label{MaxCC}
\begin{align}
    \max_{X,Y} &\; \sum_{i \in V} X_{i} &\label{MaxCC:obj}\\
    \text{s.t.} & \;\sum_{\{i,j\} \in E} Y_{ij} = \sum_{i \in V} X_{i} - 1 \label{MaxCC:tree}\\
    &\; \sum_{i \in V^c} X_{i} \le 1 & \forall\; c \in C \label{MaxCC:color}\\
    &\; \sum_{\{i,j\} \in E(U)} Y_{ij} \le |U|-1 & \forall\; U \subseteq V \label{MaxCC:subtour}\\
    &\quad  Y_{ij} \le X_i, \quad Y_{ij} \le X_j & \forall\; \{i,j\}\in E \label{MaxCC:yx2}\\
    &\quad  X \in \{0,1\}^{|V|}, \; Y \in \{0,1\}^{|E|},
\end{align}
\end{subequations}
\vspace{0.05mm}
\end{minipage}
with $E(U)$ defining the set of edges with both their endpoints in $U$. The objective function~\eqref{MaxCC:obj} gives the cardinality of the component.
Constraint~\eqref{MaxCC:tree} ensures the existence of a tree connecting all the nodes in the component (for ease of modeling, the constraint to identify a connected component is replaced by the search for an underlying tree, spanning all the nodes in the component without creating subtours).
Constraints~\eqref{MaxCC:color} guarantee the occurrence of each color at most once in the component.
Finally, constraints~\eqref{MaxCC:subtour} ensure that no subtour is contained in the selected edges.
Constraints~\eqref{MaxCC:yx2} link the $Y$ and $X$ variables, by imposing that, if $i$ or $j$ are not in the component, the edge $\{i,j\}$ is not selected either. The computational complexity of this optimization problem is unknown and exploring this aspect is an interesting direction of research.

\subsection{Tighter bound 
on the number of colorful components in a MEC optimal solution}\label{subsubsec:Q1}
We first introduce the following lemma, which states that the cardinality of the transitive closure of a component containing $n$ nodes is not smaller than the one of $k$ components whose sum of nodes is $n$.

\begin{lemma}\label{th:comp-1-to-k}
    The transitive closure of a component $S$, with $n=|S|$, contains at least as many edges as the transitive closures of $k$ components $S_1, \dots, S_k$, such that $n_i=|S_i|$, $\forall\ i \in \{1, \dots, k\}$, and $n=\sum_{i=1}^{k} n_i$, with $n, n_1, \dots, n_k \in \mathbb{N}$.
\end{lemma}

Lemma~\ref{th:comp-1-to-k} follows from the fact that the transitive closure of component $S$ can be seen as a clique $\varphi$ with $n$ nodes and $\frac{n(n-1)}{2}$ edges and, given any subcliques partition $\Phi=\{\varphi_i\}_{i=1,\dots,k}$ of $\varphi$, the number of edges in $\varphi$ is not smaller than the sum of the ones in $\varphi_1,\dots,\varphi_k$.

Before introducing the following theorem, let us observe that removing $k$ nodes from any component $S$ of $n$ nodes entails a decrease in the number of edges in the transitive closure of $S$ equal to:
{\small\begin{equation}
    \sum_{i=1}^{k} (n-i) = kn - \frac{k(k+1)}{2}.
\end{equation}}
Analogously, adding $k$ nodes to a component $S$ of $n$ nodes entails an increase equal to:
{\small\begin{equation}
    \sum_{i=1}^{k} (n-1+i) = kn -k + \frac{k(k+1)}{2}.
\end{equation}}

The following theorem gives a tighter bound on the maximum number of colorful components associated with any MEC optimal solution, i.e., a valid value of $\bar{Q}$.

\begin{theorem}\label{th:ub}
If $S$ is a colorful component of a graph $\mathcal{G}$ (not necessarily of maximum cardinality), then $\bar{Q}=|V \setminus S|+1$ is an upper bound on the number of colorful components in any MEC optimal solution.
\end{theorem}
\begin{proof}{Proof.}
Proof in Online Appendix~B.
\end{proof}

We now prove the following theorem that is related to Theorem~\ref{th:ub} and refers to the case in which two or more maximum-cardinality colorful components are available. 

\begin{theorem}\label{th:ub_MEC}
Given $k\geq 2$ maximum-cardinality disjoint colorful components of $\mathcal{G}$, hereinafter denoted by $S_i$ for $i=1,\dots, k$, with $n=|S_i|$ for all $i$, then $\bar{Q}=|V \setminus \left\{\bigcup_i S_i\right\}|+k$ is an upper bound on the number of colorful components in any MEC optimal solution.
\end{theorem}
\begin{proof}{Proof.}
Proof in Online Appendix~C.
\end{proof}

Note that in Theorem~\ref{th:ub_MEC}, contrary to Theorem~\ref{th:ub}, the components have to be of maximum cardinality, otherwise the result does not hold.
Note also that this theorem gives an upper bound $\bar{Q}$ on the number of colorful components by considering a MEC solution containing $k$ maximum-cardinality disjoint colorful components, regardless of the value of $k$. The tightest upper bound is related to the solution associated with the largest value of $k$. However, any solution provides a valid upper bound. Therefore, we propose Algorithm~\ref{algo:heu1},  which heuristically finds a sequence of maximum-cardinality disjoint colorful components by iteratively solving formulation~\eqref{MaxCC} with the additional constraint imposing that the cardinality of the component is equal to $n$. At each iteration $i$, a subgraph of $\mathcal{G}$ is considered by removing the already-found components $S_0, S_1, \dots S_{i-1}$. We also note that Algorithm~\ref{algo:heu1} provides a MEC feasible solution.

\setlength{\intextsep}{2pt}
\begin{algorithm}[ht]\caption{Computing a feasible solution and $\bar{Q}$ for MEC problem}
    \LinesNumbered
    \SetAlgoLined 
    \label{algo:heu1}
    \KwData{Graph $\mathcal{G}$.} 
    Solve formulation~\eqref{MaxCC}, obtaining the maximum-cardinality colorful component $S_0$ of $\mathcal{G}$.\\
    Set $n = |S_0|$, $V = V\setminus S_0$, $\mathcal{G} = \mathcal{G}[V]$ and $\bar{\mathcal{S}}= \{S_0\}$. \\
    \While{$|V| \geq n$}
    {Solve $(P)$ defined as formulation~\eqref{MaxCC} with the additional constraint $\sum_{i \in V} X_i = n$.\\
    \lIf{$(P)$ is not feasible}{\textit{break}.\DontPrintSemicolon}
    Let $S$ be the maximum-cardinality colorful component of $\mathcal{G}$ corresponding to the optimal solution of $(P)$. \\
    Set $\bar{\mathcal{S}}= \bar{\mathcal{S}} \cup \{S\}$, $V = V\setminus S$ and $\mathcal{G} = \mathcal{G}[V]$.\\}
    \lForEach{$u \in V$}{
        Set $S = \{u\}$ and  $\bar{\mathcal{S}} = \bar{\mathcal{S}} \cup \{S\}.$\DontPrintSemicolon
    }
    Set $\bar{Q}=|\bar{\mathcal{S}}|$.\\
    \Return $\bar{\mathcal{S}},\bar{Q}$
\end{algorithm}

\subsection{Tighter bound on the number of colorful components in a MCC optimal solution}\label{subsubsec:Q2}
An upper bound on the number of colorful components of any MCC optimal solution, with respect to the trivial bound $Q=|V|$, is obtained by computing any feasible solution for the problem. We thus propose the following heuristic algorithm, Algorithm~\ref{algo:heu2}, which computes a \textit{non-trivial} solution (a trivial solution is the one composed by $|V|$ singletons) for the MCC problem.
\setlength{\intextsep}{2pt}
\begin{algorithm}[ht]\caption{Computing a feasible solution and $\bar{Q}$ for MCC problem}
    \LinesNumbered
    \SetAlgoLined 
    \label{algo:heu2}
    \KwData{Graph $\mathcal{G}$.} 
    Set $\bar{\mathcal{S}}=\emptyset$.\\
    \While{$|V| > 0$}
    {Solve formulation~\eqref{MaxCC}, obtaining the maximum-cardinality colorful component $S$ of $\mathcal{G}$.\\
    Set $\bar{\mathcal{S}}=\bar{\mathcal{S}} \cup \{S\}$, $V=V\setminus S$ and $\mathcal{G} = \mathcal{G}[V]$.\\}
    Set $\bar{Q}=|\bar{\mathcal{S}}|.$\\
    \Return $\bar{\mathcal{S}},\bar{Q}$
\end{algorithm}

In the same vein as Algorithm~\ref{algo:heu1}, the algorithm computes the colorful component of maximum cardinality (line~3) and removes it from the graph (line~4), until the graph is empty. In this way, a sequence of colorful components of non-increasing cardinality is obtained, together with a bound $\bar{Q}$.
\section{Branch-and-cut algorithm}\label{sec:bec}
In this section, we present the branch-and-cut algorithm we use to solve the formulations presented above. The general scheme is similar across all problems, with differences related to valid inequalities and preprocessing techniques. For all formulations, connectivity constraints and path constraints are added dynamically, and the corresponding separation algorithms are presented in Section~\ref{subsec:separation}. Section~\ref{subsec:validineq} presents valid inequalities used to strengthen formulations. Some of them are specific to one problem only while others are valid for all problems. Finally, in Section~\ref{subsec:preprocessing}, we describe a preprocessing technique for the MOP problem.

\subsection{Separation of connectivity and path constraints}\label{subsec:separation}
We start describing the procedure we use to separate the aggregated connectivity constraints of type~\eqref{connectivity} or their linearized version~\eqref{connectivity_lin} while solving formulation~\eqref{MOP} for the MOP problem (we recall that in this case they are valid inequalities), \eqref{MEC} for the MEC problem and \eqref{MCC} for the MCC problem, respectively. The procedure also works for the disaggregated version of the constraints, i.e., constraints~\eqref{connectivity:disag} and~\eqref{connectivity:disag_lin}.

Connectivity constraints are separated on integer solutions only and the separation procedure, whose pseudo-code is reported in Algorithm~5 in Online Appendix~D, works as follows. It takes as input the original graph $\mathcal{G}$ and an integer solution $(\bar{{x}},\bar{{y}})$ (or $(\bar{{x}},\bar{{y}}, \bar{{z}})$  for the linearized version) and adds to the model any constraint of type~\eqref{connectivity} or \eqref{connectivity:disag} violated by $(\bar{x},\bar{y})$ (or of type~\eqref{connectivity_lin} or \eqref{connectivity:disag_lin} violated by $(\bar{x},\bar{y},\bar{z})$).
After computing the set of components of the support graph $\bar{\mathcal{G}}$, the algorithm checks whether, for each component $\bar S$, there exists a node not belonging to $\bar S$ which has been assigned to the same colorful component $k$ as a node in $\bar S$. If this is the case, the corresponding violated connectivity constraint is added to the formulation.

Path constraints~\eqref{path-ineqs} (or their linearized version~\eqref{path-ineqs_lin}) are also separated on integer solutions. The separation procedure, described in Algorithm~6 in Online Appendix~D, begins similarly to Algorithm~5 by computing the connected components of the support graph $\bar{\mathcal{G}}$. For any pair of nodes belonging to the same component, but assigned to different colorful components, all elementary paths connecting the two nodes are identified. For each such path, a violated path constraint is added to the model.

\subsection{Valid inequalities and optimality cuts}\label{subsec:validineq}
In this section, we introduce several valid inequalities, which are used to strengthen either the MOP, MEC, or MCC formulations presented in Sections~\ref{subsec:MOP}, \ref{subsec:MEC} and~\ref{subsec:MCC}. 
Some of them are valid for all feasible solutions, while others cut off parts of the feasible domain due to symmetries and dominance conditions.

\subsubsection{Symmetry-breaking inequalities}\label{subsub:symmetry}
The following symmetry-breaking inequalities can be alternatively added to the formulations presented above. The first type of inequalities orders the indices of $\mathcal{Q}$ on the basis of the cardinality of the components:
\begin{subequations}
\begin{align}
    \label{simmetries:type-1}
    \sum_{i \in V} x_{i}^k \ge \sum_{i\in V} x_{i}^{k+1}, &\qquad \forall\; k \in \{1, \dots, Q-1\}.
\end{align}\\
The second type of inequalities requires that each node $i$ belongs to a component $k$ such that $k \leq i$:
\begin{align}
    \label{simmetries:type-2}
    \sum_{k \in \{1, \dots, i\}} x_i^k = 1, &\qquad \forall\; i \in V.
\end{align}
\end{subequations}

\subsubsection{Valid inequalities on edges connecting nodes in the same colorful component}\label{subsub:edges}
Let us consider a pair of nodes $i, j \in V$ which have been assigned to the same colorful component $k \in \mathcal{Q}$. If $\{i, j\} \in E$, on the one hand, selecting such an edge may only increase the value of the MOP objective function; on the other hand, for the MEC and MCC problems, there exists an optimal solution in which edge $\{i,j\}$ is not removed. This follows from the fact that the number of edges in the transitive closure of the graph (i.e., the MEC objective value), as well as the number of colorful components (i.e., the MCC objective value), is not affected by the selection of any edge linking nodes already assigned to the same colorful component.
Thus, the following valid inequalities can be added to  formulations~\eqref{MOP}, \eqref{MEC} and \eqref{MCC}:
\begin{equation}\label{valid-N}
    y_{ij} \ge x_{i}^{k}x_{j}^{k} \qquad \forall\; \{i,j\} \in E, k \in \mathcal{Q}.
\end{equation}

\subsubsection{Optimality cuts on the minimum number of edges for MEC and MCC problems}
Here we present some cuts on the minimum number of edges belonging to the optimal solutions of the MEC or the MCC problems. 
Concerning the MEC problem, let us consider the colorful component $S$ of $\mathcal{G}$ of maximum cardinality, obtained through formulation~\eqref{MaxCC}. Then, the minimum number of edges connecting the nodes in $S$, i.e., $|S|-1$, is a lower bound on the number of edges belonging to the optimal solution of the MEC problem. 
Hence, the following optimality cut can be added to formulation~\eqref{MEC}:
\begin{equation}\label{MEC:valid-E}
    \sum_{\{i,j\} \in E} y_{ij} \ge |S|-1.
\end{equation}

Indeed, let us assume that the optimal solution of the MEC problem has at most $|S|-2$ edges. According to Lemma~\ref{th:comp-1-to-k}, and similarly to the proof of Theorem \ref{th:ub_MEC} (Online Appendix~C), the maximum value of the transitive closure, when having $|S|-2$ edges, is obtained by considering a single component of size $|S|-1$. The value of the transitive closure is $(|S|-1)(|S|-2)/2$, which is smaller than $|S|(|S|-1)/2$, i.e., the number of edges in the transitive closure of $S$. This proves that, in any optimal solution, there are at least $|S|-1$ edges.

When considering the MCC problem, given any feasible sequence $S_1, \dots, S_k$ of $k$ colorful components, we can impose the following cut on the number of edges in an optimal solution:
\begin{equation}
    \label{MCC:valid-E}
    \sum_{\{i,j\}\in E} y_{ij} \ge \sum_{i=1}^{k} (|S_i|-1).
\end{equation}

In fact, any solution having less than $k$ connected components includes at least $\sum_{i=1}^{k} (|S_i|-1)$ edges. Also in this case, as proposed for the bound on the cardinality of $\mathcal{Q}$ in Section~\ref{subsubsec:Q2}, we can heuristically determine the sequence of colorful components $S_1, \dots, S_k$ with decreasing maximum cardinality through Algorithm \ref{algo:heu2}, and use the solution obtained to tighten the bound on the number of edges. 

\subsection{Preprocessing procedure for the MOP problem}\label{subsec:preprocessing}
When minimizing the number of edges to remove, a preprocessing procedure can be applied to derive a set of optimality cuts, related to edges that can be removed a priori, and accordingly reduce the size of an instance of the MOP problem. This procedure leverages one of the rules proposed by \cite{bruckner2012}, which relies on the concept of $t$-edge-connectivity here recalled.
\begin{definition}
    A component $\mathcal{G}[S]$ is $t$-edge-connected if, for each pair of nodes $i,j \in S$, there exist at least $t$ edge-disjoint paths in $\mathcal{G}$ connecting $i$ and $j$.
\end{definition}

In the following, we denote by \textit{edge-connectivity} $t_S$ of $\mathcal{G}[S]$, the largest $t$ for which $\mathcal{G}[S]$ is $t$-edge-connected, which corresponds to the minimum number of edges to be removed from $\mathcal{G}[S]$ to disconnect it.
The rule proposed by  \cite{bruckner2012} reads as follows.
\begin{lemma}[From Rule~2 in \cite{bruckner2012}]
    \label{lemma:rule-2}
    Given a minimal edge cut $B$ of $\mathcal{G}$, with $|B|=t$, partitioning such graph into two connected components $\mathcal{G}_B$ and $\mathcal{G} \setminus \mathcal{G}_B$, if $\mathcal{G}_B$ is colorful, $t$-edge-connected and contains all the colors associated with the nodes in $H = \{ v \in \mathcal{G} \setminus \mathcal{G}_B: \exists\ u \in \mathcal{G}_B, \{u,v\} \in B\}$, i.e., the set of nodes incident with some edge in $B$ but not in $\mathcal{G}_B$, then there exists an optimal solution in which all the edges from $B$ are removed.
\end{lemma}
\begin{figure}[h]
  \caption{Example of minimal edge cut according to Lemma~\ref{lemma:rule-2}. The edges in the cut are $\{2,3\}, \{7,8\}$.}
   \centering
   \hfill
   {\includegraphics[scale=0.6]{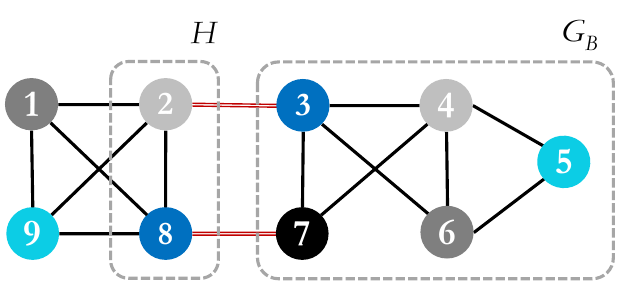}\label{fig:Preprocessing}}
  \hspace*{\fill}%
\end{figure}
\vspace{4mm}
Figure~\ref{fig:Preprocessing} shows an example of graph $\mathcal{G}$ of 9 nodes, for which $B = \{\{2,3\}, \{7,8\}\}$ is an edge cut of size $t=2$ satisfying all the properties of Lemma~\ref{lemma:rule-2}. Indeed, by removing $B$ from $\mathcal{G}$, two connected components are left, namely $\mathcal{G}[\{1,2,8,9\}]$ and $\mathcal{G}[\{3,4,5,6,7\}]$. Although both components are colorful and 2-edge-connected, Lemma~\ref{lemma:rule-2} holds only when $\mathcal{G}_B = \mathcal{G}[\{3,4,5,6,7\}]$. In this case, indeed, $H = \{2,8\}$ and all the colors associated with a node in $H$ also appear in $\mathcal{G}_B$.

To identify the largest minimal edge cut matching all the properties required by Lemma~\ref{lemma:rule-2}, we devise an integer program.
Similarly to what has been done for formulation~\eqref{MaxCC}, we identify two disjoint connected components by searching for two underlying trees, each one spanning all the nodes belonging to the same component.
In order to formulate the program, we define the following sets of binary variables:
\begin{itemize}
    \item $\alpha_i$, for each $i \in V$, s.t. $\alpha_i=1$ iff node $i$ belongs to $V_B$, namely, the set of nodes belonging to $\mathcal{G}_B$;
    \item $\beta_i$, for each $i \in V$, s.t. $\beta_i=1$ iff node $i \in H$;
    \item $\gamma_{ij}$, for each $\{i,j\} \in E$, s.t. $\gamma_{ij}=1$ iff edge $\{i,j\}$ belongs to the spanning tree associated with component $\mathcal{G}_B$ or component $\mathcal{G} \setminus \mathcal{G}_B$.
\end{itemize}
Let $N(i)$ be the set of neighbors of node $i\in V$. The formulation reads:
{\small\begin{subequations}\label{mip:rule-2}
    \begin{align}
    \max\limits_{\alpha,\beta,\gamma} &\sum\limits_{\{i,j\} \in E} \left( \alpha_{i}\beta_{j} + \alpha_{j}\beta_{i} \right) &\label{mip:rule-2-obj}\\
    \text{s.t.} &\; \sum_{i \in V} \alpha_i \ge 1, \quad \sum_{i \in V} \beta_i \ge 1 &\label{mip:rule-2-nonemptyN}\\
    &\; \beta_i \le 1-\alpha_i, \quad \beta_i \le \sum_{j \in N(i)} \alpha_j, \quad \beta_i \ge \frac{1}{|N(i)|} \sum_{j \in N(i)} \alpha_j - \alpha_i & \forall\ i \in V \label{mip:rule-2-betasetting1}\\\
    &\; \frac{1}{|V^c|}\sum_{i \in V^c} \beta_{i}  \le \sum_{i \in V^c} \alpha_{i} \le 1 & \forall\; c \in C \label{mip:rule-2-Gb-colorful}\\
    & \;\sum_{\{i,j\} \in E} \gamma_{ij} = |V|-2 \label{mip:rule-2:trees}\\
    &\; \sum_{\{i,j\} \in E(U)} \gamma_{ij} \le |U|-1 & \forall\; U \subseteq V \label{mip:rule-2:subtour}\\
    &\; \alpha_i + \alpha_j - 2 \alpha_i \alpha_j \le 1-\gamma_{ij} & \forall\ \{i,j\} \in E \label{mip:rule-2:linked-nodes}\\
    &\; \sum_{i \notin W} \alpha_i + \sum_{i \in W} (1-\alpha_i) \ge 1 & \forall\ W \subseteq V : t_W \le |\delta(W)| \label{mip:rule-2:t-edge-connectivity}\\
    &\; {\alpha} \in \{0,1\}^{|V|}, \; {\beta} \in \{0,1\}^{|V|}, {\gamma} \in \{0,1\}^{|E|}.
    \label{R2:subtour}
  \end{align}
\end{subequations}}
The objective function~\eqref{mip:rule-2-obj} maximizes the number of edges with one endpoint in $V_B$ and the other in $H$, corresponding to the size of the edge cut $B$. Constraints~\eqref{mip:rule-2-nonemptyN} prevent $V_B$ and $H$ from being empty, while constraints~\eqref{mip:rule-2-betasetting1} are imposed to correctly set the values of the $\beta$ variables, according to the connections between $\mathcal{G}_B$ and $H$: a node $i$ cannot belong to $H$ if it is in $\mathcal{G}_B$ or it has no neighbor in $\mathcal{G}_B$; conversely, $i$ must belong to $H$ if $\alpha_i=0$ and at least one neighbor of $i$ belongs to $\mathcal{G}_B$.
The right-hand side of constraints~\eqref{mip:rule-2-Gb-colorful} ensures the colorfulness of component $\mathcal{G}_B$, while the left-hand side requires that $\mathcal{G}_B$ contains all the colors associated with the nodes in $H$.
Constraints~\eqref{mip:rule-2:trees} and~\eqref{mip:rule-2:subtour} allow for the selection of exactly $|V|-2$ edges of $\mathcal{G}$, without originating cycles, which results in the identification of two disjoint trees, designated to span all the nodes in $\mathcal{G}_B$ and $\mathcal{G} \setminus \mathcal{G}_B$, respectively.
To this aim, constraints~\eqref{mip:rule-2:linked-nodes} impose that nodes linked by a selected edge belong to the same component.
Finally, if a component $\mathcal{G}[W]$ has edge-connectivity $t_W$ smaller than the size of the associated edge cut $\delta(W)$, such a component can not be selected as $\mathcal{G}_B$ and is then excluded through a no-good-cuts of type~\eqref{mip:rule-2:t-edge-connectivity}.
These cuts are separated on integer solutions only, while Algorithm~6 described by \cite{Matula1987} is used to check $t$-edge-connectivity.

The objective function~\eqref{mip:rule-2-obj} and the constraints~\eqref{mip:rule-2:linked-nodes} contain bilinear terms involving $\alpha$ and $\beta$. These terms can be linearized using techniques such as McCormick reformulation or other specialized methods. However, many of these advanced techniques are already embedded in modern solvers. Therefore, we directly provide the compact model~\eqref{mip:rule-2} to the solver used in our experiments.

Algorithm~\ref{algo:preprocessing-MOP} illustrates the preprocessing procedure proposed for the MOP problem, which relies on Lemma~\ref{lemma:rule-2} and consists in iteratively solving formulation~\eqref{mip:rule-2} to identify a largest minimal edge cut $B$, together with an associated colorful component $\mathcal{G}_B$ that will belong to the solution to the original MOP problem. In particular, at each iteration, given a solution $(\bar \alpha, \bar \beta, \bar \gamma)$ of formulation~\eqref{mip:rule-2}, then $B=\{\{i,j\} \in E : (\bar \alpha_i = 1 \land \bar \beta_j = 1) \lor (\bar \alpha_j = 1 \land \bar \beta_i = 1) \}$, and $\mathcal{G}_B = \mathcal{G}[S]$ with $S=\{i \in V: \bar \alpha_i = 1 \}$.
Graph $\mathcal{G}$ is updated by removing the identified colorful component $\mathcal{G}_B$ and the edges in $B$, which are contextually added to the set $E'$. At the end of the computation, $\mathcal{G}$ represents the preprocessed graph, while $E'$ contains all the removed edges. The value of the MOP objective function associated with the original graph can be obtained by solving the problem on the preprocessed graph, and then summing up the resulting objective function value and the number of edges contained in the colorful components $\mathcal{G}_B$ removed at each iteration of Algorithm~\ref{algo:preprocessing-MOP}.
\setlength{\intextsep}{1pt}
\begin{algorithm}[t]
\DontPrintSemicolon
\LinesNumbered
\SetAlgoLined 
\SetKwInput{Input}{Input}
\SetKwInOut{Output}{Output}
\Input{Graph $\mathcal{G}$ and set $\mathcal{Q}$.}
Set $E' = \emptyset$.\\ 
Solve formulation~\eqref{mip:rule-2} on graph $\mathcal{G}$, identifying a feasible cut $B$ and a colorful component $\mathcal{G}_B$. \\
\While{$B \neq \emptyset$}{
    Set $\mathcal{G} = \mathcal{G} \setminus \mathcal{G}_B$ and $E' = E' \cup B$.\\
    Solve formulation~\eqref{mip:rule-2} on graph $\mathcal{G}$, identifying a feasible cut $B$ and a colorful component $\mathcal{G}_B$. 
    }
\Return{$\mathcal{G}, E'$.}
\caption{Preprocessing procedure for the MOP problem}
\label{algo:preprocessing-MOP}
\end{algorithm}

\section{Computational results}\label{sec:results}
This section is dedicated to the analysis of the computational performance of the proposed mathematical formulations, in their linearized versions, examining the effect of enhancing them with the bounds, valid inequalities, warm-start, and preprocessing procedures described above.
We implemented all formulations in Python 3.10 and solved them through the Gurobi solver (version 10.0.2). All the experiments were conducted in single-thread mode, on a 3.40GHz Intel Intel(R) Core(TM) i7-3770 CPU with 16 GB RAM, by imposing a one-hour time limit and 10 GB memory limit for every run. While Python is generally slower than compiled languages, our profiling indicated that the overhead introduced by Python callbacks was negligible compared to the total solver runtime. Nonetheless, for scenarios requiring extremely high performance, re-implementing critical components in a compiled language could be considered as a potential avenue for making the solution approach even more efficient.

Source codes, benchmark instances and detailed computational results are available at the IJOC GitHub software repository associated with this paper \citep{ColorfulComponentsGithubRepo} for reproducibility and further analysis.

\subsection{Benchmark instances}
To evaluate the effectiveness of the proposed linearized formulations, we tested the instances used in \cite{bruckner2012}, generated by the authors from multiple alignment instances of the BAliBASE 3.0 benchmark \citep{Thompson2005}. Furthermore, for the instances consisting of multiple connected components, we solved the problem separately for each of them and restricted the analysis to all graphs having between 10 and 210 nodes (so as to keep the number of $z$ variables below $10^7$), resulting in a dataset of 409 instances that can be accessed at the IJOC GitHub software repository \citep{ColorfulComponentsGithubRepo}.

\subsection{Computational results}
In Online Appendix~E, we compare the performance of the two versions of the linearized connectivity constraints~\eqref{connectivity_lin} and~\eqref{connectivity:disag_lin} in Table~E.1 and the performance of the linearized edge constraints~\eqref{y_vars_lin} and the linearized path constraints~\eqref{path-ineqs_lin} in Table~E.2.
Provided the corresponding results, in the subsequent analysis, aimed at evaluating the benefit of the proposed valid inequalities and bounds, we will consider first the MEC and MCC formulations with aggregated connectivity constraints~\eqref{connectivity_lin} and the edge constraints~\eqref{y_vars_lin}, and then the MOP formulation without any (redundant) connectivity constraints and, again, the edge constraints~\eqref{y_vars_lin}.

Tables~\ref{tab:MEC-results} and \ref{tab:MCC-results} report the performances of the MEC and the MCC formulations, respectively, with different combinations of the bound provided in Section~\ref{subsec:Q3}, the valid inequalities presented in Section~\ref{subsec:validineq}, as well as the warm-start procedures discussed in Section~\ref{subsubsec:Q1} for the MEC and \ref{subsubsec:Q2} for the MCC problem.
Each row of the table reports the average results over the whole set of 409 instances for a given configuration. The first column (\textit{Model}) indicates the enabled inequalities, the second and third columns ($Q$ and \textit{Warm-start}) report information about whether Algorithm~\ref{algo:heu3} is used to compute an upper bound $\bar{Q}$. If such an algorithm is not used, the trivial bound $|V|$ is reported, otherwise, the same procedure is also used to provide a warm-start solution to the solver.
The subsequent columns report, for each configuration: 
the number of optimally solved instances (\textit{\#Opt}); the number of instances for which at least a feasible solution has been identified within the imposed time and memory limits (\textit{\#Feas}), where the number in parenthesis indicates the number of instances for which the memory limit was reached; the average lower and upper bound values at termination (\textit{LB} and \textit{UB}); the average percentage gap at termination (\textit{Gap}), computed as $(UB-LB)/UB$; the average runtime in seconds (\textit{Time}); and the average number of explored nodes of the branch-and-cut tree (\textit{\#Nodes}).
We remark that the value reported in the \textit{Time} column also includes the runtime needed for warm-start and preprocessing procedure, if included.\\
\begin{table}[htbp]
    \caption{\textbf{Performances of MEC formulations.}} 
    \label{tab:MEC-results}%
  \centering
    \resizebox*{0.9\textwidth}{!}{
    \begin{tabular}{l|cc|rrrrrrrr}
    \toprule
          \textit{Model} & $Q$ & \textit{Warm-start} & \textbf{\#Opt} & \textbf{\#Feas} & \textbf{LB} & \textbf{UB} & \textbf{Gap} & \textbf{Time} & \textbf{\#Nodes}\\
    \hline
    \textbf{MEC} & $|V|$ & - & 321   & 403 (6)     & 28.89 & 258.72 & 12.5\% & 840.90 & 41306.22 \\ \hline
    \textbf{MEC}+\eqref{valid-N}+\eqref{MEC:valid-E} & $|V|$ & - & 317   & 403 (6)     & 28.64 & 254.71 & 13.6\% & 891.95 & 54592.58\\   \hline 
    \textbf{MEC}+\eqref{simmetries:type-1} & $|V|$ & - & 356   & 407 (2)     & 25.81 & 211.21 & 5.7\% & 498.73 & 253.11\\
    \textbf{MEC}+\eqref{simmetries:type-2} & $|V|$ & - & 397   & 408 (1)     & 28.72 & 189.32 & 2.4\% & 148.19 & 100.21\\ \hline
    \textbf{MEC} & Alg.~\ref{algo:heu1} &  Alg.~\ref{algo:heu1} & 372   & 408 (1)     & 32.39 & 190.19 & 2.6\% & 429.27 & 5977.18\\
    \textbf{MEC}+\eqref{simmetries:type-2} &  Alg.~\ref{algo:heu1} & Alg.~\ref{algo:heu1} & 394   & 409 (0)     & 33.17 & 191.36 & 2.0\% & 164.20 & 163.15\\ 
    \bottomrule
    \end{tabular}%
    }
\end{table}%
\vspace{2mm}

\begin{table}[htbp]
    \caption{\textbf{Performances of MCC formulations.}} 
    \label{tab:MCC-results}%
  \centering
  \resizebox*{0.9\textwidth}{!}{
    \begin{tabular}{l|cc|rrrrrrrr}
    \toprule
          \textit{Model} & $Q$ & \textit{Warm-start} & \textbf{\#Opt} & \textbf{\#Feas} & \textbf{LB} & \textbf{UB} & \textbf{Gap} & \textbf{Time} & \textbf{\#Nodes}\\
    \hline
    \textbf{MCC} & $|V|$ & - & 387   & 406 (3)     & 4.71  & 6.98  & 2.1\% & 229.93 & 4084.24 \\ \hline
    \textbf{MCC}+\eqref{valid-N}+\eqref{MCC:valid-E} & $|V|$ & - & 384   & 406 (3)     & 4.78  & 7.09  & 2.4\% & 227.95 & 1066.66 \\ \hline
    \textbf{MCC}+\eqref{simmetries:type-1} & $|V|$ & - & 385   & 406 (3)     & 4.72  & 6.82  & 2.4\% & 237.91 & 289.33\\
    \textbf{MCC}+\eqref{simmetries:type-2} & $|V|$ & - & 399   & 408 (1)     & 4.90  & 6.13  & 1.2\% & 101.80 & 453.40\\ \hline
    \textbf{MCC} & Alg.~\ref{algo:heu2} & Alg.~\ref{algo:heu2}  & 390   & 408 (1)     & 4.87  & 5.55  & 0.9\% & 174.11 & 1875.37\\ 
    \textbf{MCC}+\eqref{simmetries:type-2} & Alg.~\ref{algo:heu2} & Alg.~\ref{algo:heu2}  & 397   & 409 (0)     & 5.02  & 5.19  & 0.4\% & 113.61 & 1290.71\\
    \bottomrule
    \end{tabular}%
    }
\end{table}%
\vspace{4mm}

In the second row of Tables~\ref{tab:MEC-results} and \ref{tab:MCC-results}, inequalities~\eqref{valid-N} are added to the model together with \eqref{MEC:valid-E} for the MEC problem, and \eqref{MCC:valid-E} for the MCC problem. Despite these additional inequalities, the obtained configurations show an increase in the average gap and a reduction in the number of instances solved to optimality. 
Indeed, including inequalities~\eqref{valid-N} and~\eqref{MEC:valid-E} in the MEC formulation, and~\eqref{valid-N} and~\eqref{MCC:valid-E} in the MCC formulation, significantly alters the cutting plane generation automatically performed by Gurobi, leading to fewer cuts and contributing to deteriorating the performance. While this yields a few additional optimally solved instances (7 for the MEC and 4 for the MCC), it also increases the number of instances exceeding the time limit (11 for the MEC and 7 for the MCC).
While for the MCC problem the average runtime is slightly improved, for the MEC problem the additional inequalities slow down the computation. 
Adding either symmetry-breaking inequalities~\eqref{simmetries:type-1} or~\eqref{simmetries:type-2} improves the performance of the MEC formulation w.r.t.\ the plain model in terms of the number of optimally solved instances, average runtime, and average gap, with inequalities~\eqref{simmetries:type-2} being the best between the two. This is not true for the MCC problem, where only inequalities~\eqref{simmetries:type-2} have a positive impact, whereas inequalities~\eqref{simmetries:type-1} worsens the computational results. Finally, when using Algorithms~\ref{algo:heu1} for the MEC and \ref{algo:heu2} for the MCC problem, all the considered metrics improved. In particular, with the configuration involving inequalities~\eqref{simmetries:type-2}, at least a feasible solution is found for all instances and both the average runtime and the average gap are remarkably reduced.\\

\begin{figure}[ht]
  \caption{{Number of instances optimally by MEC formulations within a given runtime.}}
   \centering
   \hspace*{\fill}%
   \hfill
  {\includegraphics[scale=0.59]{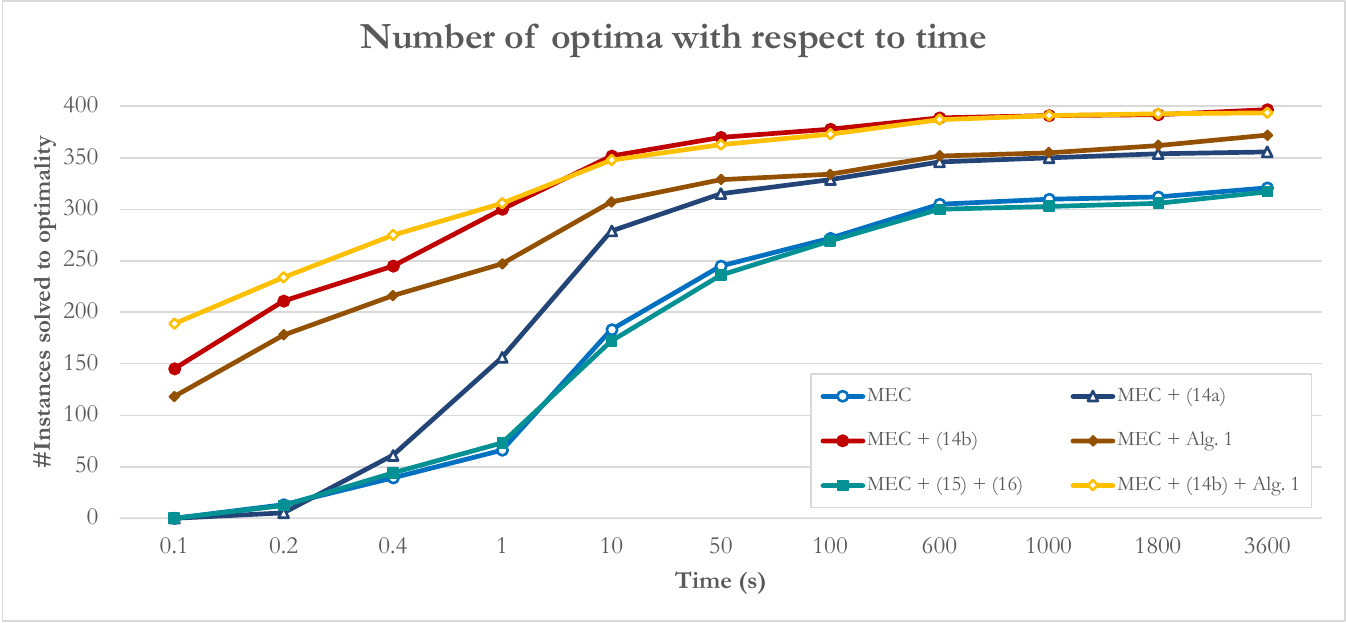}\label{fig:MEC-Performance-Chart}}\hfill
  \hspace*{\fill}%
\end{figure}
\vspace{2mm}
\begin{figure}[ht]
  \caption{{Number of instances optimally by MCC formulations within a given runtime.}}
   \centering
   \hspace*{\fill}%
   \hfill
  {\includegraphics[scale=0.59]{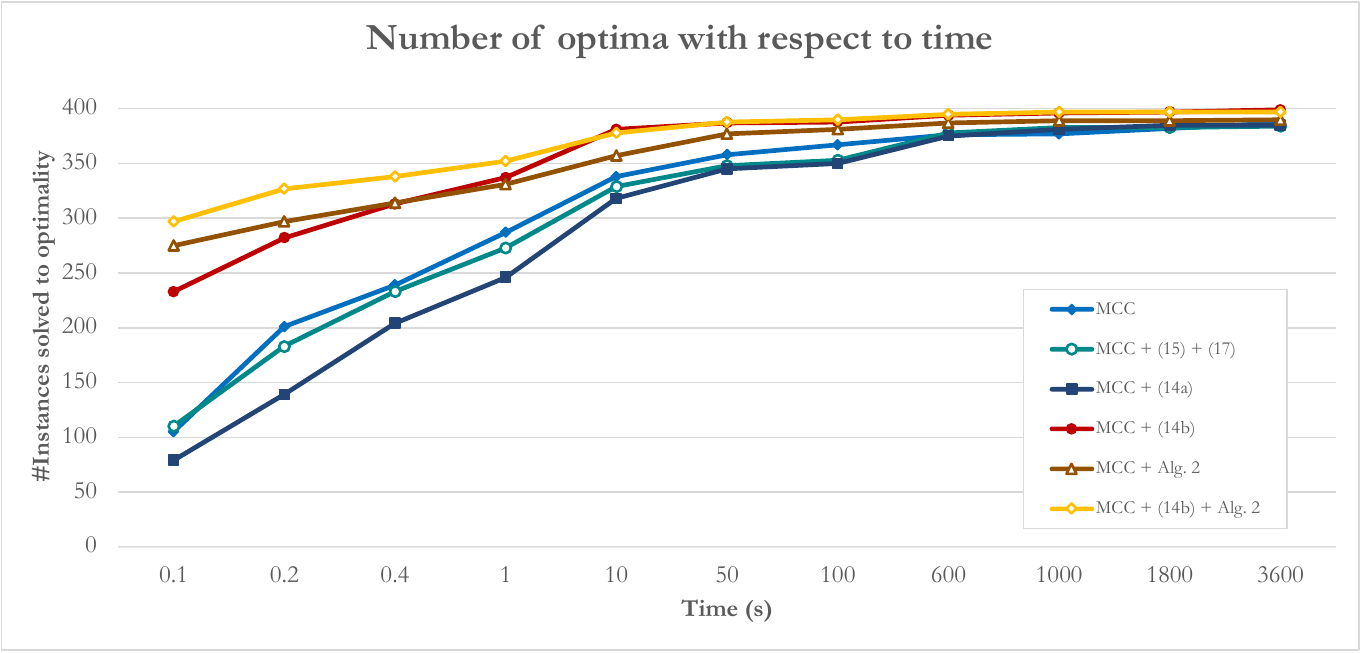}\label{fig:MCC-Performance-Chart}}\hfill
  \hspace*{\fill}
\end{figure}
\vspace{2mm}

We further provide in Figures~\ref{fig:MEC-Performance-Chart} and \ref{fig:MCC-Performance-Chart} two summary charts reporting the number of instances solved to optimality within a given computational time by all the different algorithm configurations we tested for solving the MEC and MCC problems. They illustrate how incorporating the discussed valid inequalities and algorithms enhances performance, with the configurations involving symmetry-breaking inequalities~\eqref{simmetries:type-2} and the algorithms providing a bound $\bar{Q}$ and a warm-start solution consistently outperforming the others by solving more instances in less time. These visualizations further confirm the effectiveness of the procedures and inequalities proposed.\\

\begin{table}[htbp]
    \caption{{Performances of MOP formulations.}} 
    \label{tab:MOP-results}%
  \centering
  \resizebox*{0.9\textwidth}{!}{
    \begin{tabular}
    {l|c|cc|rrrrrrrr}
    \toprule
          \textit{Model} & \textit{Preprocessing} & $Q$ & \textit{Warm-start} & \textbf{\#Opt} & \textbf{\#Feas} & \textbf{LB} & \textbf{UB} & \textbf{Gap} & \textbf{Time} & \textbf{\#Nodes} \\
    \hline
    \textbf{MOP} 	& - & $|V|$ 	& - & 408   	& 408 (1) & 18.72 & 19.32 & 0.24\% & 2.94	 	& 5.14	 	\\ \hline
    \textbf{MOP}+\eqref{valid-N} & - & $|V|$ 	& - & 408 & 408 (1) & 18.72 & 19.32 & 0.24\% & 3.22	 	& 7.24	 	\\ \hline
    \textbf{MOP}+\eqref{simmetries:type-1} 	& - & $|V|$ 	& - & 394   	& 408 (1) & 17.92 & 19.74 & 1.08\% & 158.13	& 722.36	\\
    \textbf{MOP}+\eqref{simmetries:type-2} 	& - & $|V|$ 	& - & 406   	& 408 (1) & 18.72 & 19.33 & 0.25\% & 29.61& 151.26	\\ \hline
    \textbf{MOP} & - & Alg.~\ref{algo:heu3} 	& Alg.~\ref{algo:heu3} 	& 409  & 409 (0) & 19.16 & 19.16 & 0.00\% & 4.81 & 10.55\\ 
    \textbf{MOP} & Alg.~\ref{algo:preprocessing-MOP} & $|V|$ 	& - & 408 & 408 (1) & 18.72 & 19.32 & 0.24\% & 2.03 	& 0.86\\     
    \textbf{MOP} 	& Alg.~\ref{algo:preprocessing-MOP} & Alg.~\ref{algo:heu3}  & Alg.~\ref{algo:heu3} 	& 409 & 409 (0) & 19.16 & 19.16 & 0.00\% & 4.08	& 1.98\\
    \bottomrule
    \end{tabular}%
    }
\end{table}\vspace{3mm}

As regards the MOP formulation, we report the solutions obtained by testing the different configurations in Table~\ref{tab:MOP-results}. The headings of this table are the same as in Tables~\ref{tab:MEC-results} and \ref{tab:MCC-results}, with the additional column related to the preprocessing procedure discussed in Section~\ref{subsec:preprocessing}.
The results show that the plain model, associated with the first row of the table, solves 408 instances to optimality in an average runtime of 2.94 seconds and with an average gap at termination of 0.24\%. 
Contrary to what we observed for the MEC and MCC problems, the use of the valid inequalities presented in Sections~\ref{subsub:symmetry} and~\ref{subsub:edges} does not speed up the solution process. Indeed, enabling inequalities~\eqref{valid-N} produces the same number of optimal solutions, as well as the same average lower and upper bound values, with a slightly larger average runtime due to the larger number of explored nodes of the branch-and-cut tree during the solution process.
Similarly, when adding the symmetry-breaking inequalities~\eqref{simmetries:type-1} and~\eqref{simmetries:type-2}, the number of explored nodes considerably increases, leading to a larger average gap and less instances solved to optimality.
More in detail, including inequalities~\eqref{simmetries:type-1} in the MOP formulation leads to a smaller presolved model but is also associated with a significant decrease in the number of generated cutting planes, which affects the overall performance. In contrast, including equalities~\eqref{simmetries:type-2} does not reduce the presolved model size, and is associated with a considerable slowdown of the root relaxation solution time.
Conversely, the solution process benefits from using Algorithm~\ref{algo:heu3} and/or Algorithm~\ref{algo:preprocessing-MOP}.
When the bound on the number of colorful components together with a warm-start solution (computed by  Algorithm~\ref{algo:heu3}) is used, the MOP formulation manages to optimally solve all the instances, in a average runtime of 4.81 seconds. 
When the preprocessing procedure described in Algorithm~\ref{algo:preprocessing-MOP} is employed, the solution is produced faster than the plain model, but there is still one instance that is not solved to optimality. 
Finally, we test the model with both Algorithm~\ref{algo:heu3} and Algorithm~\ref{algo:preprocessing-MOP}, obtaining the best results in terms of number of optimally solved instances and average gap, associated with a smaller average runtime compared with the other configuration solving the whole set of instances to optimality.

\begin{figure}[ht]
  \caption{{Number of instances optimally by MOP formulations within a given runtime.}}
   \centering
  {\includegraphics[scale=0.5]{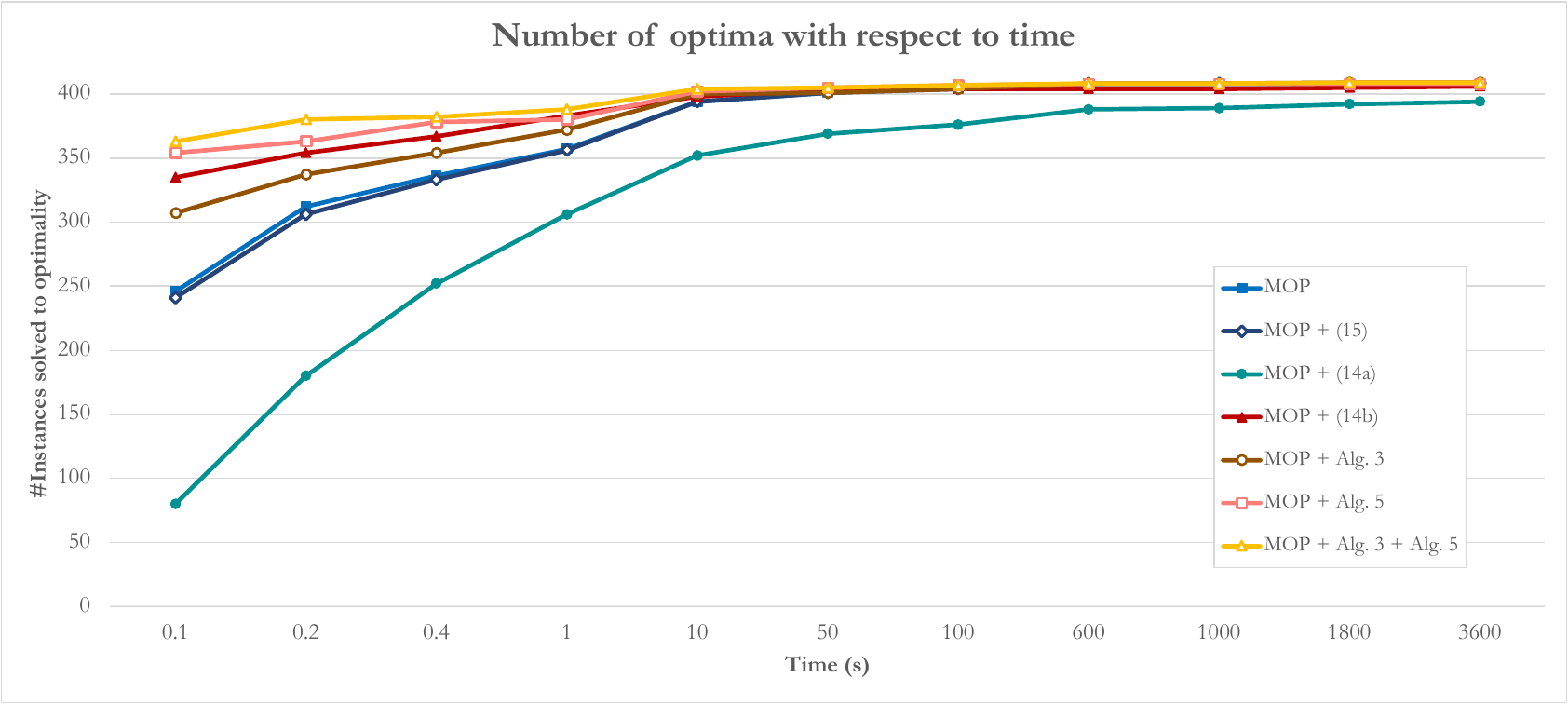}\label{fig:MOP-Performance-Chart}}
\end{figure}

The summary chart reported in Figure~\ref{fig:MOP-Performance-Chart} helps visualize the performances of the configurations w.r.t.\ the computational time, confirming that the best configuration is the one without connectivity constaints and using both the preprocessing and the warm-start procedures, together with the bound $\bar{Q}$.

\section{Conclusion}\label{sec:conclusion}
We propose integer non-linear programming formulations for three problems belonging to the class of partitioning a colored graph into colorful components, namely, the MOP, MEC, and MCC problems. The formulations are then linearized through standard techniques. An exact branch-and-cut algorithm is developed for each problem, building upon the linearized formulations and enhanced through different speed-up techniques, i.e., valid inequalities, bounds on the number of variables, warm-start heuristics, and a preprocessing procedure. All the techniques proved to be effective in improving the performance of the exact algorithms. Tests on benchmark instances show that the algorithm can solve reasonably sized instances.

To the best of our knowledge, this is the first work proposing an exact algorithm for the problems tackled. Given the relevant applications related to the problems, this work can pave the way for future research related to either strengthening the performance of the approach proposed in this work (for example by devising new classes of valid inequalities), or designing scalable heuristic approaches. In the second case, the exact approach proposed in this work can serve as a benchmark to measure the quality of the solutions provided by the heuristic.

\paragraph{\textbf{Acknowledgments:}\\}
The authors would like to thank Diego Delle Donne (ESSEC Business School in Paris) for his valuable suggestions, as well as the anonymous referees for the careful and insightful review of the manuscript.

\bibliographystyle{elsarticle-harv}
\bibliography{colorfulcomponents}
\end{document}